\documentclass[letterpaper,10pt,leqno]{article}

\usepackage{amsmath}
\usepackage{amsthm}
\usepackage{amssymb}
\usepackage{enumitem}
\usepackage{accents}
\usepackage{mathrsfs}
\usepackage{ifthen}
\usepackage{hyperref}
\usepackage{calc}

\newtheorem{thm}{Theorem}
\newtheorem{claim}{Claim}
\newtheorem{lem}{Lemma}
\newtheorem{cor}[lem]{Corollary}
\newtheorem{prop}[lem]{Proposition}
\newtheorem{question}{Question}
\newtheorem{conjecture}{Conjecture}
\newtheorem{emp}{\empstring}

\theoremstyle{definition}
\newtheorem{defn}[lem]{Definition}
\newtheorem{example}[lem]{Example}
\newtheorem{note}[lem]{Note}

\theoremstyle{remark}
\newtheorem{remark}[lem]{Remark}

\newenvironment{enumeq}[1]{
\begin{enumerate}[label=(\arabic*), ref=\arabic*, widest=#1]
\setcounter{enumi}{\value{equation}}}
{\setcounter{equation}{\value{enumi}}\end{enumerate}}

\setenumerate{leftmargin=*, label=\tu{(\alph*)}, ref=\alph*, widest=b}

\makeatletter

\newcommand{\ch}{\mathrm{CH}}

\newcommand{\pstar}{(*)}

\newcommand{\dom}{\operatorname{dom}}

\newcommand{\power}{\P}

\newcommand{\inv}{^{-1}}

\newcommand{\ext}{^\smallfrown\hspb}
\newcommand\bigext{^\frown}

\renewcommand{\div}{\mathbin{/}}

\newcommand{\imsucc}{\operatorname{imsucc}}

\newcommand{\id}{\operatornamewithlimits{id}}

\newcommand{\@shade}[2]{\nabla_{#1#2}}
\newcommand{\@shadow}[2]{\Delta_{#1#2}}
\newcommand{\shadeto}[1]{\@shade{\to}{#1}}
\newcommand{\shadowto}[1]{\@shadow{\to}{#1}}

\providecommand{\restriction}{\mathbin{\upharpoonright}}

\newcommand{\@@eq}[2]{=_{#1}^{#2}}
\newcommand{\@@neq}[2]{\ne_{#1}^{#2}}
\newcommand{\@@l}[2]{<_{#1}^{#2}}
\newcommand{\@@nl}[2]{\nless_{#1}^{#2}}
\newcommand{\@@le}[2]{\le_{#1}^{#2}}
\newcommand{\@@nle}[2]{\nleq_{#1}^{#2}}
\newcommand{\@@equiv}[2]{\equiv_{#1}^{#2}}

\newcommand{\@in}[1]{\in_#1}
\newcommand{\@eq}[1]{=_#1}
\newcommand{\@le}[1]{\le_{#1}}
\newcommand{\@l}[1]{<_#1}
\newcommand{\@nle}[1]{\nleq_#1}
\newcommand{\@ge}[1]{\ge_#1}

\newcommand{\u@le}[1]{\le^#1}
\newcommand{\u@l}[1]{<^#1}
\newcommand{\u@eq}[1]{=^#1}

\newcommand{\@lefnt}[1]{\@@le#1*}
\newcommand{\@nlefnt}[1]{\@@nle#1*}
\newcommand{\@eqfnt}[1]{\@@eq#1*}

\newcommand{\lefnt}{\mathrel{\le^*}}

\newcommand{\len}[1]{\@le{#1}}
\newcommand{\lln}[1]{\@l#1}
\newcommand{\gen}[1]{\@ge{{#1}}}

\newcommand{\str}{\mathrm{str}}

\newcommand{\asym}{\mathrm{asym}}
\newcommand{\trn}{\mathrm{tran}}
\newcommand{\negstr}{-\str}

\newcommand{\strlen}[1]{\@@le{#1}\str}
\newcommand{\nstrlen}[1]{\@@nle{#1}\str}
\newcommand{\strln}[1]{\@@l{#1}\str}
\newcommand{\nstrln}[1]{\@@nl{#1}\str}
\newcommand{\strequiv}[1]{\@@equiv{#1}\str}

\newcommand{\lenn}[2]{\@@le{#1}{#2}}
\newcommand{\nlenn}[2]{\@@nle{#1}{#2}}
\newcommand{\lnn}[2]{\@@l{#1}{#2}}
\newcommand{\nlnn}[2]{\@@nl{#1}{#2}}

\newcommand{\nleu}[1]{\nleq^{#1}}

\newcommand{\eqnn}[2]{\@@eq{#1}{#2}}
\newcommand{\neqnn}[2]{\@@neq{#1}{#2}}
\newcommand{\asymle}{\u@le\asym}
\newcommand{\asyml}{\u@l\asym}

\newcommand{\strtrnlen}[1]{\@@le{#1}{\str,\trn}}
\newcommand{\strtrnln}[1]{\@@l{#1}{\str,\trn}}
\newcommand{\trnle}{\u@le{\trn}}
\newcommand{\trnl}{\u@l{\trn}}
\newcommand{\trneq}{\u@eq{\trn}}
\newcommand{\negstrlen}[1]{\@@le{#1}\negstr}
\newcommand{\nnegstrlen}[1]{\@@nle{#1}\negstr}
\newcommand{\nnegstrln}[1]{\@@nl{#1}\negstr}
\newcommand{\negstreq}[1]{\@@eq{#1}\negstr}

\renewcommand{\aa}{\mathrm{aa}}
\newcommand{\eqaa}{\@eq\aa}
\newcommand{\inaa}{\@in\aa}

\newcommand{\leaa}{\@le\aa}
\newcommand{\nleaa}{\@nle\aa}
\newcommand{\geaa}{\@ge\aa}
\newcommand{\lefntaa}{\@lefnt\aa}
\newcommand{\nlefntaa}{\@nlefnt\aa}
\newcommand{\eqfntaa}{\@eqfnt\aa}

\newcommand{\naa}{\mathrm{naa}}
\newcommand{\eqnaa}{\@eq\naa}
\newcommand{\eqfntnaa}{\@eqfnt\naa}
\newcommand{\lenaa}{\@le\naa}
\newcommand{\lefntnaa}{\@lefnt\naa}

\newcommand{\eqae}{\@eq\ae}
\newcommand{\leae}{\@le\ae}

\newcommand{\forces}{\mskip 5mu plus 5mu\|\hspace{-2.5pt}{\textstyle \frac{\hspace{4.5pt}}{\hspace{4.5pt}}}\mskip 5mu plus 5mu}

\newcommand{\Iff}{\espc\mathrm{iff}\espc}

\newcommand{\impls}{\espc\mathrm{implies}\espc}

\providecommand{\And}{\espc\mathrm{and}\espc}

\renewcommand{\and}{\hespc\mathrm{and}\hespc}

\newcommand{\Or}{\hespc\mathrm{or}\hespc}

\renewcommand{\b}{\mathfrak b}

\newcommand{\cntum}{\mathfrak c}

\newcommand{\two}{\{0,1\}}

\newcommand{\pN}{\power(\N)}
\newcommand{\reals}{\mathbb R}

\newcommand{\N}{\mathbb N}

\newcommand{\irrational}{\irrationals}

\newcommand{\cantor}{\two^{\hsp\mathbb N}}

\newcommand{\irrationals}{\mathbb N^{\hsp\mathbb N}}
\newcommand{\irri}[1]{\N^{\hsp#1}}

\newcommand{\twoseq}{\two^{<\mathbb N}}

\newcommand{\seq}[1]{{#1}^{<\N}}

\DeclareFontFamily{U}{cmsy}{}
\DeclareFontShape{U}{cmsy}{m}{n}{<12> sfixed * [10] cmsy10 
<10> <9> <8> <7> <6> <5> sfixed * [10] cmsy10}{}
\DeclareSymbolFont{customtwo}{U}{cmsy}{m}{n} 
\DeclareMathSymbol{\sctn}{\mathord}{customtwo}{"78}

\DeclareFontFamily{U}{cmmi}{}
\DeclareFontShape{U}{cmmi}{m}{n}{<20> sfixed * [11] cmmib10 <12> sfixed * [10]
cmmi10 <10> <9> <8> sfixed * [6] cmmi6 <5> <6> <7> sfixed * [5] cmmi5}{}
\DeclareSymbolFont{custom}{U}{cmmi}{m}{n}
\DeclareMathSymbol{\rharpoon}{\mathord}{custom}{"2A}
\newlength{\widt}
\newlength{\widttwo}
\newlength{\hgt}

\newcommand{\real}{\reals}

\newcommand{\espc}{\quad}
\newlength{\@@hespc}
\newcommand{\hespc}{\hspace{\@@hespc}}
\newcommand{\overbar}[1]{\,\overline{\!{#1}}}

\newcommand{\Th}{{^{\mathrm{th}}}}

\renewcommand{\ae}{\mathrm{ae}}

\newcommand{\tu}{\textup}

\newcommand{\A}{\mathcal A}
\newcommand{\B}{\mathcal B}
\newcommand{\C}{\mathcal C}
\newcommand{\D}{\mathcal D}

\newcommand{\F}{\mathcal F}
\newcommand{\G}{\mathcal G}
\newcommand{\Hcal}{\mathcal H}

\renewcommand{\P}{\Pcal}
\newcommand{\Pcal}{\mathcal P}
\newcommand{\Q}{\mathcal Q}

\newcommand{\U}{\mathcal U}

\newcommand{\X}{\mathcal X}

\newcommand{\Fb}{\mathbb F}

\newcommand{\Tb}{\mathbb T}

\newcommand{\hsp}{\mspace{1.5mu}}
\newcommand{\hspb}{\mspace{-1.5mu}}
\newcommand{\spc}{\,\,\,}

\newcounter{saveenumi}
\newcommand{\save}{\setcounter{saveenumi}{\value{enumi}}}
\newcommand{\restore}{\setcounter{enumi}{\value{saveenumi}}}

\renewcommand{\And}{\hespc\mathrm{and}\hespc}

\newcommand{\tree}{\Tb}

\newcommand{\fpf}{\Fb}

\newcommand{\empstring}{Dichotomy}

\begin{document}

\title{Nonhomogeneous analytic families of trees}
\author{James Hirschorn}
\date{}
\maketitle

\renewcommand{\thefootnote}{}

\footnotetext{\emph{Dates}. First version: September 14,
  2003. Revised: August 10, 2008.}
\footnotetext{2000 \emph{MSC}. 
Primary 03E15; Secondary 03E40, 05D05, 28A12.}
\footnotetext{This research was primarily supported by Lise Meitner Fellowship, 
 Fonds zur F\"orderung der wissenschaftlichen Forschung, Project No.~M749-N05;
 the first version was completed 
 with partial support of  Consorcio Centro de Investigaci\'on Matem\'atica, Spanish
 Government grant No.~SB2002-0099.}

\renewcommand{\thefootnote}{\arabic{footnote}}

\begin{abstract}
We consider a dichotomy for analytic families of trees stating that either there is a
colouring of the nodes for which all but finitely many levels
of every tree are nonhomogeneous, 
or else the family contains an uncountable antichain. This dichotomy
implies that every nontrivial Souslin poset satisfying the countable chain condition
adds a splitting real. 

We then reduce the dichotomy to a conjecture of Sperner Theory.
This conjecture is concerning the asymptotic behaviour of the product
of the sizes of the $m$-shades of pairs of cross-$t$-intersecting families.
\end{abstract}

\section{Introduction}
\label{sec:introduction}

The results in this paper are directed at the general question of just how
fundamental the Cohen and random forcing notions are. More specifically we are
interested in the following question from~\cite{MR1303493}.

\begin{question}[Shelah]
\label{q-1}
Does $\{\textup{Cohen},\textup{random}\}$ form a basis for the collection of all nontrivial Souslin ccc
posets\textup{?}
\end{question}

\noindent What came as a surprise, was a deep connection between this
fundamental question in the theory of set theoretic forcing and extremal combinatorics.

 Recall that a poset is \emph{Souslin} if it can be represented a pair
$(P,\le)$ where $P\subseteq\real$ is analytic, and the partial ordering $\le$ and
the incompatibility relation are both analytic subsets of
$\real\times\real$. Clearly the Cohen and random posets are both Souslin. By
\emph{nontrivial} we mean that the poset adds a new generic object to the ground
model; in other words, the poset has a condition with no atoms below it. A
\emph{basis} for a class $\C$ of posets, is a subclass $\B\subseteq\C$ such that
every member $Q$ of $\C$ has a member $P$ in $\B$ which \emph{embeds} into it,
i.e.~there is a map $e:P\to Q^*$, where $Q^*$ is the completion of $Q$ (i.e.~the
complete Boolean algebra of regular open subsets of $Q$), such that $p\le q\to e(p)\le e(q)$
and $e$ preserves maximal antichains; equivalently,
adding a generic object for $Q$ adds a generic for $P$. 

The easier half of this question has been answered
in~\cite{MR1303493}:

\begin{thm}[Shelah, 1994]
\label{u-5}
Every Souslin ccc poset which adds an unbounded real also adds a Cohen real.
\end{thm}

Thus the Cohen poset embeds into every Souslin ccc poset which is not weakly
distributive, and thus it remains (for a positive answer) to show that every
nontrivial weakly distributive Souslin poset with the ccc adds a random
real. However, the remaining part of Question~\ref{q-1} seems to be considerably
more difficult. Indeed, many consider it a dubious conjecture that 
the collection $\{$Cohen,
random$\}$ does in fact form a basis.  Veli\v ckovi\'c has suggested some test
questions for this conjecture, including the following one from~\cite{MR1903857}.

\begin{question}[Veli\v ckovi\'c]
\label{q-2}
Does every nontrivial Souslin ccc poset add a splitting real\textup{?}
\end{question}

Recall that a real $s\subseteq\N$ is \emph{splitting} over some
model $M$, if both $s$ and its complement intersect every infinite $E\subseteq\N$ in
$M$, i.e.~$s\cap E\ne\emptyset$ and $s^\complement\cap
E\ne\emptyset$. 
Hence adding a splitting real means adding a splitting real over the
ground model. 
This is a good test question because Cohen and random reals are both also splitting reals. 
One can see this by noting that $\{x\subseteq\N:x\subseteq y\}$ and
$\{x\subseteq\N:x\cap y=\emptyset\}$ are both null (for the Haar measure on
$\pN\cong\cantor$) and meager for every infinite $y\subseteq\N$. But this fact also
follows immediately from the more general corollary~\ref{o-1} below.

We shall formulate a general
dichotomy of descriptive set theory for analytic families of trees (dichotomy~\ref{u-1}),
which has a positive answer to question~\ref{q-2} as a straightforward
consequence. Whether or not this dichotomy is true remains
unsolved. The main result of this paper, however, is the surprising
discovery of a deep connection between this dichotomy and Sperner
Theory. In particular, the famous Erd\H os--Ko--Rado Theorem
(cf.~\cite{MR0140419}, but actually proved in 1938) 
on the maximum size of $t$-intersecting families is
relevant. More specifically, in the paper~\cite{Hirs2},
we arrived at a conjecture on the
asymptotic behaviour of the maximum of the product of the size of the
$m$-shades of a pair of cross-$t$-intersecting families. The main
result of this paper is a substantial reduction of the dichotomy to this
reasonable conjecture of extremal combinatorics.\footnote{A previous
  version of this paper (\cite{Hirs}) made a stronger combinatorial
  conjecture, and established the dichotomy as a consequence. However,
  that conjecture was disproved in~\cite{Hirs2}.}

There have been some relevant developments 
since the original version of this paper was written. First of all,
question~\ref{q-2} was solved by Veli\v ckovi\'c in~\cite{MR2166361}, where he
moreover proved the following result:

\begin{thm}[Veli\v ckovi\'c, 2003]
\label{u-4}
$\pstar_{\cntum}$ implies that every nontrivial weakly distributive
ccc poset adds a splitting real.
\end{thm}

\noindent $\pstar_{\cntum}$ is the standard $P$-ideal dichotomy on
sets of size continuum (see~e.g. \cite{MR1809418}).  

However, Veli\v ckovi\'c's theorem does not seem to say much on
whether our dichotomy is true. In fact, even if Shelah's question has a
positive answer, it still does not seem to entail our dichotomy.
On the other hand, our approach of extracting a statement of descriptive
set theory is the approach used by Shelah in proving
theorem~\ref{u-5}; this has advantages because the descriptive set
theoretic statement gives an actual description of the splitting real
(specifically, in equation~\eqref{eq:47} below) as opposed to merely
asserting its existence, and the descriptive set theory is of interest
in its own right (e.g.~the descriptive set theory extracted from the
proof of theorem~\ref{u-5} is developed further in~\cite{MR1610563} 
and~\cite{MR1903857}).

The other developments have been on the measure theory front. 
In~\cite{MR2186722} an old problem of von Neumann was settled. From this one
could deduce that a positive answer to the famous Maharam's Conjecture
of abstract measure theory would imply a positive answer to Shelah's
question (question~\ref{q-1}). The most astounding development however
was Talagrand's recent negative solution to Maharam's
Conjecture~\cite{MR2214604}. While Maharam's Conjecture was considered by some
to be the most significant problem in abstract measure theory,
Shelah's question is now the natural place to look for even more
challenging measure theoretic problems in a similar vein. 

\section{A dichotomy for analytic families of trees}
\label{sec:dich-analyt-famil}

We write $\N=\{0,1,\dots\}$ for the nonnegative integers. 
For a parameter $f\in\irrationals$ we let $\tree(f)$ denote the family of all subtrees of
the tree 
\begin{equation}
  \label{eq:34}
  \seq f=\bigcup_{n=0}^\infty\prod_{i=0}^{n-1}\{0,\dots,f(i)-1\}
\end{equation}
of all functions $t$ with domain $\{0,\dots,n-1\}$ for some $n\in\N$, with
$t(i)<f(i)$ for all $i<n$, ordered by inclusion. To avoid trivialities we assume
that $f\ge\underline 2$, i.e.~%
\begin{equation}
  \label{eq:35}
  f(n)\ge 2\espc\text{for all $n$}.
\end{equation}
Its topology is of course obtained by identifying $\tree(f)$ with the product $\{0,1\}^{\seq f}$.
Note that $\seq{\underline 2}$ is the binary tree and is more usually
denoted by $\twoseq$. Indeed, this is the primary case and one will
not lose much by taking $f=\underline 2$ throughout. 

A \emph{branch} through a tree $T$ refers to an infinite downwards
closed chain of~$T$. Recall that by K\"onig's classical tree lemma,
every infinite $T\in\tree(f)$ has an infinite branch. 
We say that $T$ is \emph{infinitely branching} if
$T$ has infinitely many branches. More familiar to set theorists are
\emph{perfect trees} $T$ where every node has an extension to a
splitting node (i.e.~a node with at least two immediate successors); 
thus perfect trees are in particular infinitely branching. 

For a tree $T\in\tree(f)$, we write
$T(n)\subseteq\prod_{i=0}^{n-1}\{0,\dots,\allowbreak f(i)-1\}$ 
for the $n\Th$ level of $T$, which consists of sequences of length
$n$, and with $T(0)$ the singleton consisting of the empty sequence
whenever $T\ne\emptyset$. And we write $T\restriction n$ for the
subtree consisting of the first $n$ levels $T(0),\dots,T(n-1)$ of $T$. 
For $\epsilon=0,1$, a subset $A\subseteq \seq f$ is called
\emph{$\epsilon$-homogeneous} for some colouring $c:\seq f\to\two$,
if $c(s)=\epsilon$ for all $s\in A$. We say that $A$ is
\emph{homogeneous} for $c$ if $c$ is constant on~$A$, i.e.~iff $A$ is
either $0$-homogeneous or $1$-homogeneous for $c$; $A$ is
\emph{nonhomogeneous} for $c$ if it is not homogeneous. 

For a family of trees $\A\subseteq\tree(f)$ and some tree
$T\in\tree(f)$, we denote $\A_T=\{S\in\A:S\subseteq T\}$. 

\begin{note}
\label{note:1}
It is important to note that we use the set theoretic terminology for
antichains of a poset. I.e.~if $(P,\le)$ is a poset, then $A\subseteq P$ 
is an \emph{antichain} iff $a$ is \emph{incompatible} with $b$,
written $a\perp b$, for all $a\ne b$ in $A$, where $a\perp b$ means
that there is no $p\in P$ satisfying both $p\le a$ and $p\le b$. 

We emphasize this because in extremal combinatorics (e.g.~\cite{MR2003b:05001})
an antichain refers instead to the weaker concept of 
a pairwise \emph{incomparable} subset. These are also called
\emph{Sperner families}, and we shall use the latter terminology.
\end{note}

\begin{emp}
\label{u-1}
Every analytic family $\A$ of infinitely branching subtrees of $\seq f$
satisfies at least one of the following\textup:
\begin{enumerate}
\item\label{item:2} There exists a colouring $c:\seq f\to\two$ and a
  $T\in\A$ such
  that $S(n)$ is homogeneous for $c$ for at most finitely many
  $n\in\N$, for every $S\in\A_T$.
\item\label{item:1} The poset $(\A,\subseteq)$ has an uncountable antichain.
\end{enumerate}
\end{emp}

\begin{remark}
\label{r-2}
In a previous version (\cite{Hirs}) we had a stronger dichotomy where
there was no restriction to $T$ in the first
alternative~\eqref{item:2}, i.e.~$c$ is nonhomogeneous for every
$T\in\A$. We think that the latter stronger dichotomy is entailed by
the former; however, we wish to avoid the additional complications here.
\end{remark}

\begin{remark}
\label{r-1}
We have not yet considered
 what happens when we remove the restriction that the subtrees 
be infinitely branching. 
It is obvious that dichotomy~\ref{u-1} is false if we allow subtrees
with only one branch, and we conjecture that it is still false
even if we do not allow this triviality. 
\end{remark}

Let us see how this yields a splitting real. 

\begin{lem}
\label{o-2}
Dichotomy~\textup{\ref{u-1}} implies that every nontrivial Souslin poset with the ccc adds a splitting real.
\end{lem}
\begin{proof}
Let $\P$ be a nontrivial Souslin poset with the ccc. It is known that $\P$ must add a new
real; this is proved in~\cite{MR1303493}, and also follows from the fact that this statement
is absolute, since one can construct a Souslin tree from any nontrivial poset
with the ccc which does not add a real (see~\cite{MR1832454}). 
Thus we can find a $\P$-name $\dot r$ for a new real in $\cantor$. For each
$p\in\P$, we let $T_p\in\tree(\underline 2)$ be the tree of possibilities for $\dot r$
which is defined by
\begin{equation}
  \label{eq:38}
  T_p=\bigl\{t\in\twoseq:q\forces \check t\subseteq\dot r\text{ for some }q\le p\bigr\}.
\end{equation}
The family $\A=\{T_p:p\in\P\}$ is an analytic subset of $\twoseq$, because $\P$ is
Souslin and since the ccc property allows for maximal antichains to be described as
reals (see~\cite{MR1303493}, or~\cite{MR1610563}). Since $\dot r$ names a new real, each $T_p$ is a
perfect tree, and thus in particular has infinitely many branches. And since
$q\le p$ implies $T_q\subseteq T_p$,  if $T_p$ and $T_q$ are incompatible in the
poset $(\A,\subseteq)$, then $p$ and~$q$ are incompatible in $\P$. 
Since the second condition~\eqref{item:1} entails the
existence of an uncountable antichain in $(\A,\subseteq)$, by the ccc
dichotomy~\ref{u-1} (with $f=\underline2$) yields a colouring
$c:\twoseq\to\two$ as in condition~\eqref{item:2}. 

Define a $\P$-name $\dot s$ for a real in $\pN$ by
\begin{equation}
  \label{eq:47}
  n\in \dot s\Iff c(\dot r\restriction n)=1.
\end{equation}
Then $\dot s$ names a splitting real, because for every $p\in\P$ and every infinite
$E\subseteq\N$, we can find an $n\in E$ such that $T_p(n)$ is nonhomogeneous
for $c$, and thus there are $t_0$ and $t_1$ in $T_p(n)$ such that 
\begin{equation}
  \label{eq:50}
  c(t_0)=0\And c(t_1)=1.
\end{equation}
And there are $q_0,q_1\le p$ forcing that $\dot r\restriction n=t_0$ 
and $\dot r\restriction n=t_1$, respectively. 
Thus $q_0\forces(\N\setminus\dot
s)\cap E\ne\emptyset$ and $q_1\forces\dot s\cap E\ne\emptyset$.
\end{proof}

\subsection{Generic splitting reals}
\label{sec:gener-splitt-reals}

Every $f\in\irrationals$ ($f\ge\underline 2$) determines a space of reals 
\begin{equation}
  \label{eq:21}
  \real(f)=\prod_{n=0}^\infty\{0,\dots,f(n)-1\}.
\end{equation}
Note that for $T\in\tree(f)$, $[T]\subseteq\real(f)$ is a closed subset in the
product topology, where $[T]$ is the set of all reals corresponding to a branch
through $T$ (i.e.~$x\restriction n\in T$ for all $n$). 
By a \emph{perfect} subset of a topological space, we mean a nonempty closed set
with no isolated points. Thus $[T]$ is perfect whenever $T$ is a
perfect tree. 

The natural measure on $\real(f)$ is the Haar probability measure $\mu_f$ determined by the
group operation on $\real(f)$ of coordinatewise addition modulo $f(n)$. Notice that
\begin{equation}
  \label{eq:22}
  \mu_f([T])=\lim_{n\to\infty}\nu_{fT}(n)
\end{equation}
where $\nu_{fT}(0)\ge \nu_{fT}(1)\ge\cdots$ is the sequence 
\begin{equation}
\label{eq:16}
\nu_{fT}(n)=\frac{|T(n)|}{\prod_{i=0}^{n-1}f(i)}.
\end{equation}

\begin{lem}
\label{l-5}
For all $f\in\irrationals$ \textup($f\ge\underline 2$\textup), there exists a colouring
$c:\seq f\to\two$ such that for every $T\in\tree(f)$ at least one of the following holds.
\begin{enumerate}
\item There are only finitely many $n$ for which $T(n)$ is homogeneous for $c$.
\item $\mu_f([T])=0$.
\end{enumerate}
\end{lem}
\begin{proof}
Define a colouring $c:\seq f\to\two$ by
\begin{equation}
  \label{eq:18}
  c(s)=0\Iff s(|s|-1)<\left\lfloor \frac{f(|s|-1)}2\right\rfloor,
\end{equation}
where $\lfloor x\rfloor$ is the largest integer $\le x$. Observe that if $T(n)$ is
homogeneous for $c$ (in either colour) for all $n$ in some finite set
$F$, then setting $n=\max(F)+1$, $\mu_f([T])\le\frac{|T(n)|}{\prod_{i=0}^{n-1}f(i)}\le
\frac{\prod_{i\in F}\lfloor\frac{f(i)+1}2\rfloor\cdot\prod_{i\notin
    F}f(i)}{\prod_{i=0}^{n-1}f(i)}\le(\frac23)^{|F|}$.
\end{proof}

We say that a subfamily $\A\subseteq\tree(f)$ \emph{separates points} if
every~$T\in\A$ and every~$x,y\in[T]$ has a $U\subseteq T$ in $\A$ such
that at most one of $x$ and $y$ is in $[U]$. It is clear that if $\A$
is a nonempty family of subtrees of $\seq f$ that separates points,
then the poset $(\A,\subseteq)$ forces a generic real in $\real(f)$. 

Obviously we are identifying $\real(\underline2)=\cantor$ 
with $\pN$ in corollary~\ref{o-1}.

\begin{cor}
\label{o-1}
If $\A\subseteq\tree(\underline 2)$ is a nonempty family that
separates points which has the property that $\mu([T])\ne0$ for all
$T\in\A$, then the generic real of the poset $(\A,\subseteq)$ is a
splitting real.
\end{cor}
\begin{proof}
Let $c:\twoseq\to\two$ be the colouring from~\eqref{eq:18}. Notice that if $\dot r$
is an $\A$-name for the generic real then, the real $\dot s$
defined in~\eqref{eq:47} satisfies
\begin{equation}
  \label{eq:10}
  n\in\dot s\Iff \dot r(n-1)=1.
\end{equation}
And by lemma~\ref{l-5} the second paragraph
of the proof of lemma~\ref{o-2} shows that $\dot s$ names a splitting real. It
follows that $\dot r$ which is a shift of $\dot s$ is also a splitting real.
\end{proof}

\section{Decay of branching}
\label{sec:domin-with-math}

The next dichotomy focuses on a lower bound for the size of the levels
rather than nonhomogeneity. 

\renewcommand{\irrational}{[\N]^\infty}

\begin{lem}
\label{u-2}
For every analytic $\A\subseteq\tree(f)$ consisting of infinitely
branching trees, at least one of the following holds\textup:
\begin{enumerate}
\item\label{item:15} There is an $h\in\irrationals$ 
with $\lim_{n\to\infty}h(n)=\infty$ such that $|T(n)|\ge h(n)$ for all but finitely many $n$, for all $T\in\A$.
\item\label{item:4} There is a perfect subset $\B\subseteq\A$ such
  that $T\cap U$ has only finitely many branches for all $T\ne U$ in $\B$.
\end{enumerate}
\end{lem}

Note that the condition~\eqref{item:4} implies that the poset $(\A,\subseteq)$ has
an antichain of size continuum. 
Lemma~\ref{u-2} is essentially proved in~\cite{MR1303493}. It also
follows immediately from the following closely related dichotomy,
which is, essentially, implicit in~\cite{MR1303493}, and is explicit,
with a different proof, in~\cite{MR1903857}. 
The deduction of lemma~\ref{u-2} from lemma~\ref{l-4} is spelled out
in~\cite{Hirs}.
Note that $\irrational$ is the family of all
infinite subsets of $\N$, and $f\lefnt g$ if $f(n)\le g(n)$ for all but finitely
many $n$. For an infinite
$x\subseteq\N$ we write $e_x\in\irrationals$ for the
strictly increasing enumeration of $x$ starting with $e_x(0)=\min(x)$.

\begin{lem}
\label{l-4}
For every analytic subset $\X\subseteq\irrational$, at least one of the following
holds\textup:
\begin{enumerate}
\item The family $\{e_x:x\in\X\}$ is bounded in $(\irrationals,\lefnt)$.
\item There is a perfect subset of $\X$ consisting of pairwise almost disjoint sets.
\end{enumerate}
\end{lem}

Let us point out a corollary of lemma~\ref{l-4} which is interesting since, for
example, every $\Sigma^1_2$ set is a union of $\aleph_1$ many analytic sets.

\begin{cor}
If $\X\subseteq\irrational$ is the union of less than $\b$ analytic sets, then either
\begin{enumerate}
\item the family of enumerating functions of members of $\X$ is bounded in $(\irrationals,\lefnt\nobreak)$, or
\item there is a perfect almost disjoint subset of $\X$.
\end{enumerate}
\end{cor}

Unfortunately, this dichotomy (lemma~\ref{u-2}) is too weak for our
needs. We need to bound the decay in branching, that is how much
smaller the levels of $S(n)$ are than $T(n)$ for $S\subseteq T$ in
$\A$, but first we must measure it. 

\subsection{Measuring the decay}
\label{sec:measuring-decay}

\newcommand{\Rln}{\len{\Hcal}}
\newcommand{\iln}{\len{\id}}
\newcommand{\rln}{\len \tau}
\newcommand{\nrln}{\nleu \tau}
\renewcommand{\irri}{\irrationals_\infty}
\newcommand{\meas}{\nu_\U^{\tau}}

Let $\irri\subseteq\irrationals$ denote the subfamily of 
functions $g$ such that $\lim_{n\to\infty}g(n)=\infty$.
For a function $\tau:[0,\infty)\to[0,\infty)$ on the nonnegative reals,
define a function $\rho_\tau:\irri\times\irri\to[0,\infty]$ by
\begin{equation}
  \label{eq:91}
  \rho_\tau(g,h)=\liminf_{n\to\infty}\frac{h(n)}{\tau\circ g(n)}.
\end{equation}
Then define a relation $\rln$ on $\irri$ by
\begin{equation}
  \label{eq:72}
  g\rln h\qquad\text{if}\qquad\rho_\tau(g,h)>0.
\end{equation}
For a family $\Hcal$ of such functions,
we define a relation $\Rln$ on $\irri$ by
\begin{equation}
  \label{eq:71}
  g\Rln h\qquad\text{if}\qquad g\rln h
  \espc\text{for all $\tau\in\Hcal$}.
\end{equation}
Typically we shall have
\begin{enumerate}[label=(\roman*), ref=\roman*, widest=iii]
\item\label{item:16} $\tau(x)=O(x)$ (i.e.~$\tau(x)$ is bounded by $Cx$ for
  some $C>0$),
\item\label{item:17} $\lim_{x\to\infty}\tau(x)=\infty$,
\item\label{item:24} $\tau$ is concave 
\save
\end{enumerate}
for all $\tau\in\Hcal$. Condition~\eqref{item:16} ensures that
\begin{equation}
  \label{eq:72a}
  g\le h\impls g\rln h,
\end{equation}
and in particular each $\rln$, and thus $\Rln$, is reflexive. 
Condition~\eqref{item:17} ensures that each $\rln$ is
nontrivial. Concavity gives us
\begin{equation}
  \label{eq:92}
  \tau(c x)\ge c\tau(x)\espc\text{for all $c\le 1$}.
\end{equation}
The reason we use a family of functions is to
obtain transitivity. A family $\Hcal$ with every $\tau\in\Hcal$
satisfying~\eqref{item:16}--\eqref{item:24}, that is moreover
closed under ``taking roots'', i.e.~
\begin{enumerate}[label=(\roman*), ref=\roman*, widest=iii]
\restore
\item\label{item:23} $\forall \tau\in\Hcal\,\exists \sigma\in\Hcal
\spc \sigma\circ \sigma=\tau$,
\end{enumerate}
is  said to be \emph{suitable}. 

\begin{example}
\label{x-2}
The singleton $\Hcal=\{\id\}$ is suitable.
\end{example}

\begin{example}
\label{x-1}
 $\Hcal=\{x\mapsto x^\alpha:0<\alpha<1\}$ is suitable since
 $x^{\sqrt\alpha}\circ x^{\sqrt\alpha}=x^\alpha$.
We could include $\alpha=1$, but then $\Rln$ would be the same as $\iln$.
\end{example}

\begin{lem}
\label{p-2}
Let $\Hcal$ be suitable. Then $\Rln$ is a quasi ordering of $\irri$. 
\end{lem}
\begin{proof}
Suppose $g_0\Rln g_1$ and $g_1\Rln g_2$. Take $\tau\in\Hcal$, and find
$\sigma\in\Hcal$ such that $\sigma\circ\sigma=\tau$.
Choose $0<\varepsilon_0<1$ so that $\varepsilon_0\rho_\sigma(g_0,g_1)\le 1$
and let $0<\varepsilon_1<1$ be arbitrary.
Then there are $k_0$ and $k_1$ such that 
$g_1(n)>\varepsilon_0\rho_\sigma(g_0,g_1)
\cdot \sigma\circ g_0(n)$ for all $n\ge k_0$ and
$g_2(n)>\varepsilon_1\rho_\sigma(g_1,g_2)\cdot\sigma\circ g_1(n)$ 
for all $n\ge k_1$. Hence, using~\eqref{eq:92} and the fact that 
$\sigma$ must be nondecreasing, 
$g_2(n)>\varepsilon_1\rho_\sigma(g_1,g_2)
\cdot\sigma\bigl(\varepsilon_0\rho_\sigma(g_0,g_1)
\cdot \sigma\circ g_0(n)\bigr)\ge 
\varepsilon_0\rho_\sigma(g_0,g_1)\varepsilon_1\rho_\sigma(g_1,g_2)
\cdot\tau\circ g_0(n)$
for all $n\ge\max(k_0,k_1)$, proving that $g_0\rln g_2$.
\end{proof}

For a subtree $T\subseteq\seq f$ with infinitely many branches,
clearly $h_T\in\irri$, were $h_T(n)$ is the size of the $n\Th$ level:
\begin{equation}
  \label{eq:93}
  h_T(n)=|T(n)|.
\end{equation}

\begin{prop}
\label{p-6}
Assume $\Hcal$ is suitable.
If $S\subseteq T$ then $h_S\le h_T$, and hence $h_S\Rln h_T$. 
\end{prop}

\begin{defn}
\label{d-1}
Let $\Hcal$ be a suitable family, and let $\A\subseteq\tree(f)$ be a
family of infinitely branching trees. We say that $\A$ has
\emph{locally bounded branching decay with respect to $\Hcal$}
if there exists $T\in\A$ such that $h_T\Rln h_S$ for all $S\subseteq
T$ in $\A$. 
We say that $\A$ has \emph{\tu(globally\tu) bounded
  branching decay with respect to $\Hcal$} if the preceding statement holds
for all $T\in\A$. For $\Hcal$ as in example~\ref{x-1}, we say that $\A$
has \emph{\tu(locally\tu) \tu[globally\tu] bounded branching decay
for $x^{<1}$}.
\end{defn}

\begin{example}
\label{x-3}
We consider the case $\Hcal=\{\id\}$. Suppose that $\A\subseteq\tree(f)$  with $\seq f\in\A$. 
If $\A$ has globally bounded branching decay with respect to $\id$,
then by definition, $h_{\seq f}\Rln h_T$ for all $T\in\A$. Observe
however, that
\begin{equation}
  \label{eq:98}
  \rho_{\id}(h_{\seq f},h_T)=\mu_f([T])\espc\text{for all $T\in\tree(f)$}
\end{equation}
(cf.~equation~\eqref{eq:22}). Hence the global bounding property is
equivalent to every member of $\A$ having positive measure. 

If $\A$ has locally bounded branching decay with respect to $\id$, then there exists
some $T\in\A$ such that $h_T\Rln h_S$ for all $S\subseteq T$ in $\A$. 
Observe that $\rho_{\id}(h_T,\cdot)$ may define  a finitely additive
measure on $\A_T$ (if $T$ has
``uniform'' branching), even though we may have $\mu_f([T])=0$. 
Indeed, in this case we can define strictly positive finitely additive
measure by $\nu(S)=\lim_{n\to\U}\frac{h_S(n)}{h_T(n)}$ for a
nonprincipal ultrafilter $\U$. 
We remark, without proof, that it is possible to generalize the
construction in equation~\eqref{eq:18} to obtain a nonhomogeneous
colouring for $\A_T$. 
\end{example}

Now we consider a family $\A$ that has locally bounded branching decay
for~$x^{<1}$, witnessed by some $T\in\A$. 
First we notice that this is a
weaker condition than having locally bounded branching decay for $\id$.

Let $\Hcal$ be the family from example~\ref{x-1}.

\begin{prop}
\label{p-7}
For all $g,h\in\irri$, for all $\tau\in\Hcal$, $\rho_{\id}(g,h)>0$
implies $\rho_\tau(g,h)=\infty$. 
\end{prop}

For reasons related to the central limit theorem (cf.~\cite{Hirs2}), we
are especially interested in $\alpha=\frac12$.

\begin{prop}
\label{p-8}
For all $g,h\in\irri$ and all $0<\alpha<1$,
 $g\Rln h$ implies that
 $\liminf_{n\to\infty}\frac{h(n)}{g(n)^\alpha}=\infty$.
\end{prop}

It is easy to construct an analytic, indeed countable, family of trees
that is not weakly distributive and does not have locally bounded
branching decay for $x^{<1}$. For this reason, our Sperner Theoretic
analysis does not apply to the non-weakly distributive case. An
analytic family of trees that is equivalent as a forcing notion to
Cohen forcing does have a nonhomogeneous colouring.
However, the existence of nonhomogeneous colourings for arbitrary
analytic non-weakly distributive ccc families of trees depends on the
more difficult problem than question~\ref{q-1}, of finding a complete
characterization of Souslin ccc forcing notions, and not just a basis. 
Thus we restrict our attention to the weakly distributive case to
obtain the following weakening of dichotomy~\ref{u-1}.

\begin{emp}
\label{dichotomy:1}
Every weakly distributive analytic family $\A$ of infinitely branching subtrees of $\seq f$
satisfies at least one of the following\textup:
\begin{enumerate}
\item\label{item:2a} There exists a colouring $c:\seq f\to\two$ and a
  $T\in\A$ such
  that $S(n)$ is homogeneous for $c$ for at most finitely many
  $n\in\N$, for every $S\in\A_T$.
\item\label{item:1a} The poset $(\A,\subseteq)$ has an uncountable antichain.
\end{enumerate}
\end{emp}

\begin{question}
\label{q-3}
Does every analytic family $\A\subseteq\tree(f)$ of infinitely
branching trees, that both is weakly distributive and satisfies the
ccc, have locally bounded branching decay for $x^{<1}$\tu?
\end{question}

We shall prove (theorem~\ref{u-3}) that a positive answer to
question~\ref{q-3} together with our Sperner Theoretic 
conjecture (conjecture~\ref{co:1})
establishes dichotomy~\ref{dichotomy:1}. 

Let us point out that the Sperner Theory is making a substantial
contribution towards the dichotomy. This is because if we ask instead
for ``locally bounded branching decay for $\id$'' in question~\ref{q-3}, the
answer is most likely negative. Indeed, in view of example~\ref{x-3},
we expect that Talagrand's construction (\cite{MR2214604}) 
of a counterexample to the Control Measure problem
provides a counterexample to this question.

\section{The coideals}
\label{sec:coideal}

For a subset $A\subseteq\N$ and $k\in\N$ we let 
\begin{equation}
\label{eq:42}
A-k=A\setminus\{0,\dots,k\},
\end{equation} 
with $A-(-1)=A$. Please note that this disagrees with many authors' usage!
For a function  $F$ with domain $\N$ denote
\begin{equation}
  \label{eq:41}
  \N(F)=\bigcup\prod_{n=0}^\infty F(n),
\end{equation}
i.e.~the collection of all pairs $(n,x)$ with $x\in F(n)$. And for a subset
$S\subseteq\N(F)$ and $k\in\N$, denote
\begin{equation}
  \label{eq:43}
  S-k=\{(n,x)\in S:n>k\}.
\end{equation}
Throughout this article we assume that $F$ is a sequence of nonempty finite sets,
and thus $\N(F)$ is countably infinite.

We recall that a set $A\subseteq\N$ \emph{diagonalizes} a sequence $(A_k:k\in\N)$ of
subsets of $\N$ if $A-k\subseteq A_k$ for all $k\in A$. And a coideal $\Hcal$ on $\N$ is \emph{selective}
if every \emph{descending} sequence $A_0\supseteq A_1\supseteq A_2\supseteq\cdots$
of members of $\Hcal$ has a diagonalization in $\Hcal$. We next generalize the definition
of selective to coideals on $\N(F)$.

\begin{defn}
$B\subseteq\N(F)$ \emph{diagonalizes} a sequence $(B_k:k\in\N)$ of subsets of~$\N(F)$ 
if $B-k\subseteq B_k$ whenever $k\in\dom(B)$. The definition of a selective
coideal on $\N(F)$ is the same as for $\N$ but with this notion of diagonalization.
\end{defn}

We suppose now that $\U$ is a nonprincipal selective
ultrafilter on $\N$, and that $g\in\irrationals$ with $g(n)$ nonzero for all $n$ is a given parameter. For a subset
$A\subseteq\N(F)$, write $A(n)=\{x:(n,x)\in A\}$. Define
$\Hcal=\Hcal(\U,F,g)$ by 
\begin{equation}
  \label{eq:1}
  \Hcal=\left\{A\subseteq\N(F):\lim_{n\to\U}\frac{|A(n)|}{g(n)}=\infty\right\}.
\end{equation}

\begin{prop}
\label{p-1}
If $\lim_{n\to\infty}\frac{g(n)}{|F(n)|}=0$ then $\Hcal$ is a nonempty selective coideal.
\end{prop}
\begin{proof}
The assumptions that the limit is zero and that $\U$ is nonprincipal ensure that
$\N(F)\in\Hcal$. Obviously $\Hcal$ is a coideal. 

Suppose that $(A_k:k\in\N)$ is a descending sequence of members of $\Hcal$. Then there
is a descending sequence $(U_k:k\in\N)$ of members of $\U$ such that 
\begin{equation}
  \label{eq:2}
  \frac{|A_{k}(n)|}{g(n)}>k\espc\text{for all $n\in U_{k}$,}
\end{equation}
for all $k\in\N$. And then there is a $U\in\U$ with $U-k\subseteq U_k$ for all
$k\in U$. Now if we define
\begin{equation}
  \label{eq:30}
  B=\bigcup_{n\in U}\Biggl(\{n\}\times\bigcap_{k:n\in U_{k}}A_k(n)\Biggr),
\end{equation}
then $B\in\Hcal$ because $(A_k:k\in\N)$ is descending. And for all $k$, 
if $k\in \dom(B)$ then $k\in U$ and thus $B-k\subseteq \bigcup_{n\in U_k}
\bigl(\{n\}\times B(n)\bigr)\subseteq A_k$.
\end{proof}

\section{The forcing notion}
\label{sec:partial-order}

Define
\begin{equation}
  \label{eq:23}
  \fpf(F)=\{a\subseteq\N(F):a\text{ is a finite partial function}\},
\end{equation}
and for $a\in\fpf(F)$, denote $\hat a=\max(\dom(a))$ with $\hat\emptyset=-1$. For
a subset $S\subseteq\N(F)$ and $a\in\fpf(F)$, we let
\begin{equation}
  \label{eq:24}
  S-a=S-\hat a.
\end{equation}
Suppose that $\Hcal$ is a nonempty selective coideal on $\N(F)$.
Let $\Q=\Q(\Hcal)$ be the poset of all pairs $(a,A)$ where $a\in\fpf(F)$ and $A\in\Hcal$,
ordered by $(b,B)\le(a,A)$ if $a\sqsubseteq b$ (i.e.~$b$ end-extends $a$),
$b\setminus a\subseteq A$ and $B\subseteq A$. An extension $(b,B)\le(a,A)$ is called
a \emph{pure extension} of $(a,A)$ if $b=a$.
For convenience, we also insist that $A-a=A$, i.e.~$\hat a<\min(\dom(A))$.

Forcing with $\Q$ introduces a generic partial function $\dot c$ on $\N$ with $\dot
c(n)\in F(n)$, by defining
\begin{equation}
  \label{eq:32}
  \dot c=\bigcup_{(a,A)\in\dot\G}a
\end{equation}
where $\dot\G$ is a $\Q$-name for the generic filter of $\Q$. We also write 
\begin{equation}
  \label{eq:52}
  \dot D=\dom(\dot c).
\end{equation}

\begin{lem}
\label{l-6}
For every sentence $\varphi$ in the forcing language of $\Q$,
every condition in $\Q$ has a pure extension deciding whether or not $\varphi$ is true.
\end{lem}
\begin{proof}
Fix $(a,A)\in\Q$ and a sentence $\varphi$. For all $b\in\fpf(F)$ and $B\in\Hcal$,
we will say that $B$ \emph{accepts} $b$ if $(b,B-b)\forces\varphi$, $B$ \emph{rejects}
$b$ if there is no $C\subseteq B$ in $\Hcal$ which accepts $b$, and $B$
\emph{decides} $b$ if it either accepts or rejects $b$. The following facts are
immediate from the definitions.
\begin{enumeq}{10}
\item\label{item:10} For all $b\in\fpf(F)$ and $B\in\Hcal$ there is a $C\subseteq B$ in $\Hcal$
  deciding $b$.
\item\label{item:11} If $B$ accepts/rejects $b$, then $C$ accepts/rejects $b$ for
  every $C\subseteq B$ in $\Hcal$.
\end{enumeq}

\begin{claim}
\label{c-4}
If $B\in\Hcal$ accepts $b\cup\{s\}$ for all $s\in B-b$, then $B$ accepts $b$.
\end{claim}
\begin{proof}
Assume that $B$ accepts $b\cup\{s\}$ for all $s\in B-b$. 
Supposing towards a contradiction that $B$ does not accept $b$, there exists
$(c,C)\le(b,B-b)$ such that
\begin{equation}
\label{eq:51}
(c,C)\forces\lnot\varphi.
\end{equation}
By extending $(c,C)$ if necessary, we can assume that
$b$ is a proper initial segment of $c$. Then there is an $s\in B-b$ with
$b\cup\{s\}\sqsubseteq c$. However, by assumption,
$(b\cup\{s\},B-\{s\})\forces\varphi$, contradicting~\eqref{eq:51} because $(c,C)\le(b\cup\{s\},B-\{s\})$.
\end{proof}

\begin{claim}
\label{c-1}
There is a $B\subseteq A$ in $\Hcal$ which decides $a\cup b$ 
for every finite partial function $b\subseteq B$.
\end{claim}
\begin{proof}
Using~\eqref{item:10}, construct a descending 
sequence $A\supseteq B_0\supseteq B_1\supseteq\cdots$ 
of members of $\Hcal$ such that 
\begin{equation}
  \label{eq:40}
  B_n\text{ decides }a\cup b\espc\text{for all partial functions $b\subseteq A$ with
    $\hat b\le n$,}
\end{equation}
for all $n$. Let $B\subseteq B_0$ in $\Hcal$ diagonalize $(B_n:n\in\N)$. Then $B$
decides $a$ by~\eqref{item:11} and~\eqref{eq:40} with $n=0$. 
Take $\emptyset\ne b\subseteq B$ in $\fpf(F)$. Then $B-(a\cup b)=B-\hat b\subseteq B_{\hat b}$ and
thus by~\eqref{eq:40}, $B$ decides $a\cup b$.
\end{proof}

Letting $B$ be as in claim~\ref{c-1}, in particular, $B$ decides $a$, and thus we
need only concern ourselves when $B$ rejects $a$, in which case we make the
following claim. 

\begin{claim}
\label{c-2}
There is a $C\subseteq B$ in $\Hcal$ which rejects $a\cup b$ for every finite partial
function $b\subseteq C$.
\end{claim}
\begin{proof}
Using claims~\ref{c-4} and~\ref{c-1}, 
we can construct a descending sequence $B\supseteq
C_0\supseteq C_1\supseteq\cdots$ in $\Hcal$ so that
\begin{multline}
  \label{eq:48}
  C_n\text{ rejects }a\cup b\cup\{s\}\espc
  \text{for all $s\in C_n-b$},\\ \text{for all $b\subseteq B$ in $\fpf(F)$ with
    $\hat b\le n$ where $B$ rejects $a\cup b$}.
\end{multline}
Let $C\subseteq C_0$ in $\Hcal$ diagonalize the sequence. It follows by induction that $C$ rejects
$a\cup b$ for all $b\subseteq C$ in $\fpf(F)$, because if $b\subseteq C$, 
$s\in C-(a\cup b)=C-\hat b\subseteq C_{\hat b}$ (where $C_{-1}=C_0$) and $C$ rejects $a\cup b$, 
then since $B$ decides $a\cup b$, $B$ rejects $a\cup b$ by~\eqref{item:11},  
and thus $C$ rejects $a\cup b\cup\{s\}$ by~\eqref{eq:48}.
\end{proof}

Now if there were some $(b,D)\le(a,C)$ forcing $\varphi$, then since $b\setminus a\subseteq C$,
this would contradict that $C$ rejects $b$. Thus $(a,C)\forces\lnot\varphi$.
\end{proof}

\begin{lem}
\label{l-1}
Let $\varphi(x)$ be a formula in the forcing language of $\Q$. Then for every
$(a,A)\in\Q$, there exists $B\subseteq A$ in $\Hcal$ such that $(a\cup b,B-b)$ decides
$\varphi(\check a\cup \check b)$ for every finite partial function $b\subseteq B$.
\end{lem}
\begin{proof}
Using lemma~\ref{l-6}, 
choose a descending sequence
 $A\supseteq B_0\supseteq B_1\supseteq\cdots$ 
in~$\Hcal$ such that 
\begin{equation}
  \label{eq:46}
  (a\cup b,B_n)\text{ decides }\varphi(a\cup b)\espc
  \text{for all $b\subseteq  A$ in $\fpf(F)$ with $\hat b\le n$}.
\end{equation}
If $B\subseteq B_0$ in $\Hcal$ diagonalizes $(B_n:n\in\N)$, then $B$ satisfies the conclusion of the lemma.
\end{proof}

\section{Shades}
\label{sec:shades}

One of the basic notions in Sperner theory is the \emph{shade} (also
called \emph{upper shadow}) of a set or a family
of sets (see e.g.~\cite{MR2003b:05001},\cite{MR1429390}). 
For a subset $x$ of a fixed set $S$, the shade of $x$ is
\begin{equation}
  \label{eq:8}
  \nabla(x)=\{y\subseteq S:x\subset y\text{ and }|y|=|x|+1\},
\end{equation}
and the shade of a family~$X$ of subsets of $S$ is 
\begin{equation}
\label{eq:11}
\nabla(X)=\bigcup_{x\in X}\nabla(x).
\end{equation}
Recall that the \emph{$m$-shade} (also called \emph{upper $m$-shadow}
or \emph{shade at the $m\Th$ level}) of $x$ is
\begin{equation}
  \label{eq:13}
  \shadeto m(x)=\{y\subseteq S:x\subseteq y\text{ and }|y|=m\},
\end{equation}
and $\shadeto m(X)=\bigcup_{x\in X}\shadeto m(x)$. We follow the
Sperner theoretic conventions of writing $[m,n]$ for the set
$\{m,m+1,\dots,n\}$ and  $[n]$ for the set $[1,n]=\{1,\dots,n\}$. 

We introduce the following notation for colouring sets with two colours. For a set
$S$, let $S \choose [m]$ denote the collection of all colourings $c:S\to\two$ with
$|c\inv(0)|=m$, i.e.~$c\inv(0)=\{j\in S:c(j)=0\}$. This is related to shades,
because for all $c\in{S\choose[m]}$ and all $x\subseteq S$,
\begin{multline}
  \label{eq:49}
  x\text{ is homogeneous for }c\\ \Iff c\inv(0)\in\shadeto m(x)\Or c\inv(1)\in\shadeto{|S|-m}(x).
\end{multline}

When a nonhomogeneous colouring is desired, it is most efficient to use colorings in
$S\choose [m]$ for $|S|=2m$. Equation~\eqref{eq:49} immediately gives us:

\begin{lem}
\label{l-2}
Suppose $X$ is a family of subsets of $[2m]$. Then
\begin{equation*}
  \left|\left\{c\in{[2m]\choose[m]}:\exists x\in X
      \spc x\textup{ is homogeneous for }c\right\}\right|
  \le2|\shadeto m(X)|
\end{equation*}
\textup(the shades are with respect to $S=[2m]$\textup).
\end{lem}

\subsection{Upper bounds}
\label{sec:upper-bounds}

Recall that a family $\A$ of sets is \emph{$t$-intersecting} if $|E\cap F|\ge t$ for all
$E,F\in\A$; and a pair~$(\A,\B)$ of families of subsets of some fixed set are \emph{cross-$t$-intersecting} if 
\begin{equation}
  \label{eq:26}
  |E\cap F|\ge t\espc\text{for all $E\in\A$, $F\in\B$}.
\end{equation}
Thus $\A$ is $t$-intersecting iff $(\A,\A)$ is cross-$t$-intersecting.

We use the standard notation $S\choose k$ to denote the collection of
all $k$-subsets of $S$, and hence $[n]\choose k$ denotes the collection
of all subsets of~$[n]$ of cardinality $k$. Let $I(n,k,t)$ denote
the family of all $t$-intersecting subfamilies of $[n]\choose k$
(where $t\le k\le n$). 
Define the function
\begin{equation}
  \label{eq:31}
  M(n,k,t)=\max_{\A\in I(n,k,t)}|\A|.
\end{equation}
The investigation into the function $M$ and the structure of the
maximal families was initiated by Erd\"os--Ko--Rado in 1938, but not
published until~\cite{MR0140419}. In this paper, they gave a complete
solution for the case $t=1$, and posed what became one of most famous
open problems in this area. The following so called
\emph{$4m$-conjecture} for the case $t=2$:
\begin{equation}
  \label{eq:39}
  M(4m,2m,2)=\frac 12 \biggl({{4m}\choose{2m}}-{{2m}\choose m}^{\hspb 2}\biggr).
\end{equation}
We briefly explain the significance of the right hand side expression.
Define families
\begin{equation}
  \label{eq:62}
  \F_i(n,k,t)=\left\{F\in{[n]\choose k}:|F\cap[t+2i]|\ge t+i\right\}
  \espc\text{for $0\le i\le\frac{n-t}2$}.
\end{equation}
Clearly each $\F_i(n,k,t)$ is $t$-intersecting. It is not hard to
compute that the cardinality of $\F_{m-1}(4m,2m,2)$ is equal to the right hand
side of equation~\eqref{eq:39}. 
Indeed, the $4m$-conjecture was
generalized by Frankl in 1978 (\cite{MR519277}) as follows: For all
$1\le t\le k\le n$,
\begin{equation}
  \label{eq:63}
  M(n,k,t)=\max_{0\le i\le\frac{n-t}2}|\F_i(n,k,t)|.
\end{equation}
In 1995, the general conjecture was proven true by
Ahlswede--Khachatrian in \cite{MR1429238}.

Since we are interesting in applying lemma~\ref{l-2}, we define for $1\le t\le k\le m\le n$,
\begin{equation}
  \label{eq:64}
  M_0(n,m,k,t)=\max_{\A\in I(n,k,t)}|\shadeto m(\A)|,
\end{equation}
i.e.~$M_0(n,m,k,t)$ is the maximum size of the $m$-shade of a
$t$-intersecting family of $k$-subsets of $[n]$.  We have that
\begin{equation}
  \label{eq:65}
  M_0(n,m,k,t)\le M(n,m,t),
\end{equation}
but this is not optimal. In~\cite{Hirs2}, we argued that the following
conjecture should be true.

\begin{conjecture}
\label{j-1}
$\displaystyle M_0(n,m,k,t)=\max_{0\le i\le\min(k-t,\frac{n-t}2)}|\F_i(n,m,t)|$.
\end{conjecture}

\noindent We then deduced from conjecture~\ref{j-1} that whenever $k$
and $t$ are functions of $m$ satisfying $k(m)=o(m)$, 
i.e.~$\lim_{m\to\infty}k(m)\div m=0$,
$\lim_{m\to\infty}k(m)=\infty$ and
$\lim_{m\to\infty}t(m)\div\sqrt{k(m)}=\infty$, 
\begin{equation}
\label{eq:9}
\lim_{m\to\infty}\frac{M_0(2m,m,k(m),t(m))}
  {{2m\choose m}}=0.
\end{equation}

\subsection{Cross-$t$-intersecting families}
\label{sec:cross-t-intersecting}

It turns out that for out application we need upper bounds on the size
of shades of cross-$t$-intersecting families (cf.~equation~\eqref{eq:26}).
Let $C(n,k,l,t)$ be the collection of all pairs $(\A,\B)$ of cross-$t$-intersecting
families, where $\A\subseteq{[n]\choose k}$ and $\B\subseteq{[n]\choose l}$. Then the cross-$t$-intersecting function
corresponding to $M$ is defined by
\begin{equation}
  \label{eq:57}
  N(n,k,l,t)=\max_{(\A,\B)\in C(n,k,l,t)}|\A|\cdot|\B|.
\end{equation}

There are a number of results on cross-$t$-intersecting families in
the literature; however, the state of knowledge seems very meager
compared with $t$-intersecting families. 
The following theorem, proved in~\cite{MR90g:05008}, is the strongest
result of its kind that we were able to find. 

\begin{thm}[Matsumoto--Tokushige, 1989]
$N(n,k,l,1)={n-1\choose k-1}{n-1\choose l-1}$ whenever $2k,2l\le n$.
\end{thm}

Generalizing $M_0$, we define the maximum size $N_0(n,m_k,m_l,k,l,t)$
of the product of the $m_k$-shade with the $m_l$-shade of a pair of
cross-$t$-intersecting families of $k$-subsets and $l$-subsets of
$[n]$, respectively:
\begin{equation}
  \label{eq:58}
  N_0(n,m_k,m_l,k,l,t)
  =\max_{(\A,\B)\in C(n,k,l,t)}|\shadeto {m_k}(\A)|\cdot|\shadeto {m_l}(\B)|. 
\end{equation}
For purposes of our dichotomy, we are exclusively interested in the numbers
$N_0(2m,m,m,k,k,t)$. Thus we define
\begin{equation}
  \label{eq:45}
  N_1(n,m,k,t)=N_0(n,m,m,k,k,t).
\end{equation}

In~\cite{Hirs2}, we justify a conjecture for the value of
$N_0(n,m_k,m_l,k,l,t)$, corresponding to conjecture~\ref{j-1}. 
We then consider the asymptotic formula for cross intersecting pairs,
corresponding to equation~\eqref{eq:9}:

\begin{conjecture}
\label{co:1}
Assume $k(m)=o(m)$ and $\lim_{m\to\infty}k(m)=\infty$. Suppose that
$\lim_{m\to\infty}\frac{t(m)}{\sqrt{k(m)}}=\infty$.~Then
\begin{equation}
\label{eq:90}
\lim_{m\to\infty}
\frac{\sqrt{N_1(2m,m,k(m),t(m))}}{{2m\choose m}}=0.
\end{equation}
\end{conjecture}

\section{The theorem}
\label{sec:dichotomy}

The main result of this paper is the following.

\begin{thm}
\label{u-3}
Assume conjecture~\textup{\ref{co:1}} is true.
Let $\A\subseteq\tree(f)$ be a weakly distributive analytic family
of infinitely branching trees satisfying the ccc that has locally
bounded branching decay for $x^{<1}$. Then there exists a colouring $c:\seq f\to\two$ and a
$T\in\A$ such that $S(n)$ is nonhomogeneous for $c$ for all but
finitely many $n$, for all $S\in\A_T$.
\end{thm}
\begin{proof}
Let $\A$ be as hypothesized. Find $T'\in\A$ such that
\begin{equation}
h_{T'}\Rln h_S\espc\text{for all $S\in\A_{T'}$}\label{eq:17}
\end{equation}
(cf.~equation~\eqref{eq:93}), where $\Hcal=\{x^{\alpha}:0<\alpha<1\}$.
By lemma~\ref{l-5} we may assume that there exists $T\subseteq T'$ in $\A$ with
$\mu_f([T])=0$. Then setting $k(n):=h_T(n)$, we have
\begin{gather}
  \label{eq:60}
  |S(n)|\le k(n)\espc\text{for all $n$, for all $S\in\A_T$},\\
  \label{eq:61}
  \lim_{n\to\infty}\frac{k(n)}{\prod_{i=0}^{n-1} f(i)}=0.
\end{gather}
Furthermore, by lemma~\ref{p-2}, proposition~\ref{p-6} and
equation~\eqref{eq:17}, $h_T\Rln h_S$ for all $S\in\A_T$, and thus by
proposition~\ref{p-8}, we have
\begin{equation}
  \label{eq:70}
  \liminf_{n\to\infty}\frac{|S(n)|}{k(n)^\alpha}=\infty\espc\text{for all
    $S\in\A_T$, for all $0<\alpha<1$}.
\end{equation}
Finally, we notice that by increasing the function $f$ if necessary, we may assume
that $f(0)$ is even.

Now we are ready to specify the parameters for the coideal. 
Define~$m\in\irrationals$ by
\begin{equation}
  \label{eq:5}
  m(n)=\frac{\prod_{i=0}^{n-1} f(i)}2,
\end{equation}
which makes sense, taking the empty product as zero, since $f(0)$ is even.
The function $F$ on $\N$ is given by
\begin{equation}
  \label{eq:3}
  F(n)={2m(n)\choose[m(n)]}\espc\text{for all $n$},
\end{equation}
where the ``$2m(n)$'' is identified with the set 
\begin{equation}
  \label{eq:53}
 \prod_{i=0}^{n-1}\{0,\dots, f(i)-1\}
\end{equation}
of cardinality $2m(n)$. Fix $\frac12<\beta<1$, and define
$g\in\irrationals$ by
\begin{equation}
  \label{eq:6}
  g(n)=\left\lceil\sqrt{N_1(2m(n),m(n),k(n),k(n)^\beta)}\right\rceil.
\end{equation}
Note that by~\eqref{eq:61}, $k(n)=o(m(n))$,
$\lim_{n\to\infty}k(n)=\infty$ since $T$ is infinitely branching and
$\lim_{n\to\infty}k(n)^\beta\div\sqrt{k(n)}
=\lim_{n\to\infty}k(n)^{\beta-\frac12}=\infty$. Therefore,
conjecture~\ref{co:1} applies, giving
$\lim_{n\to\infty}\frac{g(n)}{|F(n)|}
=\lim_{n\to\infty}\frac{\sqrt{N_1(2m(n),m(n),k(n),k(n)^\beta)}}
{{2m(n)\choose m(n)}}=0$. Thus proposition~\ref{p-1} applies.

By going to a forcing extension, if necessary, we assume that there exists a
selective nonprincipal ultrafilter $\U$ on $\N$. Note that this existence follows from~$\ch$, or
more generally if the covering number of the meager ideal is equal to the
continuum. 
We are forcing with $\Q=\Q(\Hcal(\U,F,g))$, and we claim that $\Q$
forces that for all $S\in\A_T$,
\begin{equation}
  \label{eq:7}
  S(n)\text{ is nonhomogeneous for }\dot c
  \text{ for all but finitely many $n\in\dot D$}.
\end{equation}
This will complete the proof, because the existence of $D$ and $c$ as in the first
alternative is a $\Sigma^1_2$ statement, and thus Schoenfield's Absoluteness Theorem
applies to give these two in the ground model. Once we have a
colouring that is nonhomogeneous on some fixed infinite set, this
trivially yields a colouring for which each $S\in\A_T$ is nonhomogeneous
for all but finitely many levels. 

Suppose towards a contradiction that this is false. 
Then there exists a $\Q$-name
$\dot S$ for a member of $\A_T$ and a condition $(a,A)\in\Q$ forcing that
\begin{equation}
  \label{eq:19}
  \dot S(e_{\dot D}(n))\text{ is homogeneous for }\dot c\text{ for infinitely many $n$}.
\end{equation}

Before continuing, we give a very rough idea of how a contradiction comes
about. The definition of $\Hcal(\U,F,g)$ ensures that, in the limit as
$n\to\U$, the $A(n)$'s contain a large number of colourings of the
tree of height $n$ in equation~\eqref{eq:53}. From this
and~\eqref{eq:19} we argue (in
claim~\ref{c-5}) that along some ``path'' $p$, where $p\restriction n$
decides $\dot S(n-1)$, the uncertainty in $\dot S(n)$ is large,
i.e.~the set $R(p\restriction n)$ of all finite trees $t$ such that some extension of
$(a,A)$ forces $\dot S(n)=t$ is large; more precisely, we argue that
the $m$-shade of $R(p\restriction n)$ is large. This entails that
along two different paths $p\ne q$, the product of the cardinalities
of the pair $(R(p\restriction n),R(q\restriction n))$ will be too
large to be $t(n)$-intersecting for a suitably large $t(n)$. This in
turn will entail that the two interpretations of $\dot S$ along the
paths $p$ and $q$ will have an intersection whose branching is too low
to be a member of $\A$. Hence, from uncountably many paths we obtain
an uncountable antichain. 

Note that $(b,B)$ decides $e_{\dot D}(|b|-1)=\hat b$ 
and not $e_{\dot D}(|b|)$. 

\begin{claim}
\label{c-3}
The set of all conditions $(b,B)\in\Q$ forcing that
$\dot S(e_{\dot D}(|b|))$ is homogeneous for $\dot c$ is dense below $(a,A)$.
\end{claim}
\begin{proof}
Fix $(a_1,A_1)\le(a,A)$.
By lemma~\ref{l-1} there is a $B\subseteq A_1$ in $\Hcal$ such that $(a_1\cup b,B-b)$
decides whether or not $\dot S(e_{\dot D}(|a_1\cup b|))$ is homogeneous for 
$\dot c$ for every finite partial function $b\subseteq B$. 
And by~\eqref{eq:19}, there exists $(b,C)\le(a_1,B)$ and $n\ge|a_1|$ such that $(b,C)$
forces $\dot S(e_{\dot D}(n))$ is homogeneous for $\dot c$.
By going to an extension of $(b,C)$, we can assume that $|b|\ge n$. 
Then letting $d\sqsubseteq b$ be the initial segment of length~$n$,
$(d,B-d)$ decides whether $\dot S(e_{\dot D}(|d|))$ is homogeneous for $\dot c$
because $d\setminus a_1\subseteq B$ and $a_1\sqsubseteq d$ as $|a_1|\le n$. 
However, as $(b,C)\le(d,B-d)$, it must decide this
positively. Since $(d,B-d)\le(a_1,A_1)$, the proof is complete.
\end{proof}

Let $M\prec H_{(2^{\aleph_0})^+}$ be a countable elementary submodel
containing $f$, $T$, $\U$ and $\dot S$.
Let $\overbar M$ be its transitive collapse, and
let $\D_n$ ($n\in\N$) enumerate all of the dense subsets 
of $\Q^{\overbar M}=\Q\cap \overbar M$ in $\overbar M$. 
We construct a subtree $U$ of $(\fpf(F),\sqsubseteq)$, conditions $(b_u,B_u)$ ($u\in U$)
in $\Q\cap \overbar M$, $R(u)\subseteq\seq f$ ($u\in U$)
and $n_i\in\N$ ($i\in\N$) by recursion on $|u|$ so that
\begin{enumerate}[label=(\roman*), ref=\roman*, widest=viii]
\item\label{item:8} $(b_\emptyset,B_\emptyset)\le(a,A)$, 
\item\label{item:5} $n_{i}\in\bigcap_{|u|=i}\dom(B_u)$,
\item\label{item:6} $B_u(n_i)> 2g(n_i)$ for all $|u|=i$,
\item $u\sqsubseteq v$ implies $(b_u,B_u)\le(b_v,B_v)$,
\item\label{item:9} $(b_u,B_u)\in\D_{|u|}$,
\item\label{item:3} $(b_u,B_u)$ forces that $\dot S(e_{\dot D}(|b_u|))$ 
is homogeneous for $\dot
  c$,
\item\label{item:28} $U(i+1)=\{u\ext s:u\in U(i)$, $s=(n_i,x)$, 
$x\in B_u(n_i)\}$,
\item\label{item:29} ${b_u}\bigext s\sqsubseteq b_{u\ext s}$ 
for all $u\ext s\in U(i+1)$, 
\item\label{item:7} $(b_{u\ext s},B_{u\ext s})\forces\dot S(n_i)
=R(u\ext s)$ for all $|u|=i$.
\end{enumerate}
It is possible to satisfy~\eqref{item:5} and~\eqref{item:6} by the definition of
$\Hcal(\U,F,g)$, using the fact that $\U$ is a filter. Claim~\ref{c-3} is used for~\eqref{item:3}.

\begin{claim}
\label{a-2}
$R(u\ext (n_i,x))$ is homogeneous for $x$,
for all $u\ext (n_i,x)\in U(i+1)$. 
\end{claim}
\begin{proof}
Set $s=(n_i,x)$. Since $({b_u}\bigext s,B_u)\forces e_{\dot D}(|b_u|)=n_i$,
$(b_{u\ext s},B_{u\ext s})\forces\dot S(n_i)$ is homogeneous for $\dot
c$ by~\eqref{item:3} and~\eqref{item:29}. 
And since $({b_u}\bigext s,B_u)\forces \dot c(n_i)=x$, the claim follows
from~\eqref{item:7}.
\end{proof}

\begin{claim}
\label{c-5}
For all $u\in U(i)$, 
\begin{equation}
  \label{eq:59}
  |\shadeto {m(n_i)}\{R(v):v\in\imsucc(u)\}|
  >\sqrt{N_1(2m(n_i),m(n_i),k(n_i),k(n_i)^\beta)}.
\end{equation}
\end{claim}
\begin{proof}
By lemma~\ref{l-2}, the number of colourings in
$2m(n_i)\choose [m(n_i)]$ for which some $R(v)$ is homogeneous is at most
\begin{equation}
  \label{eq:14}
  2|\shadeto {m(n_i)}\{R(v):v\in\imsucc(u)\}|.
\end{equation}
Suppose towards a contradiction that
the value in~\eqref{eq:14} is smaller than $B_u(n_i)$. Then there exists $x\in
B_u(n_i)$ for which no member of $\{R(v):v\in\imsucc(u)\}$ is homogeneous. 
Thus, putting $s=(n_i,x)$, $R(u\ext s)$ is nonhomogeneous for $x$,
contradicting claim~\ref{a-2}. 
Therefore, by~\eqref{item:6}, the value in~\eqref{eq:14} is as big as 
\begin{equation}
\label{eq:12}
B_{u}(n_i)>2\sqrt{N_1(2m(n_i),m(n_i),k(n_i),k(n_i)^\beta)}.\qedhere
\end{equation}
\end{proof}

Define a map $G$ on $\N$ where $G(i)$ is the set of all functions from
$\{0,1\}^i$ into 
\begin{equation}
  \label{eq:25}
  \prod_{u\in U(i)}\imsucc(u)
\end{equation}
for each $i$. 
Denote the space
\begin{equation}
  \label{eq:55}
  \real(G)=\prod_{i=0}^\infty G(i)
\end{equation}
and endow it with its product topology. Now let $c\in\real(G)$ be a Cohen real over
the ground model, and set $d:=\bigcup_{i=0}^\infty c(i)$; 
hence, the domain of $d$ is
$\twoseq$ and $d(s)=c(|s|)(s)\in\prod_{u\in U(|s|)}\imsucc(u)$ 
for all $s\in\twoseq$. 
In the Cohen extension, we define a mapping
$\Phi:\cantor\to\tree(f)$ as follows. 
For each $z\in\nobreak\cantor$, 
define by recursion on~$i$, $b_0(z)=\emptyset$ and
\begin{equation}
  \label{eq:27}
  b_{i+1}(z)=b_{i}(z)\ext d(z\restriction i)\bigl(b_{i}(z)\bigr)\in U(i+1).
\end{equation}
Thus for every $z\in\cantor$, $\bigcup_{i=0}^\infty b_i(z)$ is a
branch through $U$. Then let
\begin{equation}
  \label{eq:56}
  \Phi(z)=\bigcup_{i=0}^\infty R(b_i(z)).
\end{equation}

Let $\C$ be the poset directly giving the Cohen real $c$, i.e.~each
condition of $\C$ is an element of $\prod_{i=0}^{n}G(i)$ for some
$n\in\N$. 

\begin{claim}
\label{c-7}
The range of $\Phi$ is a subset of $\A_T$.
\end{claim}
\begin{proof}
Fix $z\in\cantor$, and set $x:=\bigcup_{i=0}^\infty b_i(z)$. 
By~\eqref{item:9}, $\{b_i(z):i\in\N\}$ is $\Q^{\overbar M}$-generic
over $\overbar M$. Thus, in $\overbar M[\{b_i(z):i\in\N\}]=\overbar M[x]$, 
$\dot S[x]=\Phi(z)$ by~\eqref{item:7}, where $\dot S[x]$ denotes the
generic interpretation of $\dot S$. Hence $\overbar M[x]\models\Phi(z)\in\A_T$. 
By analytic absoluteness 
between the transitive models $\overbar M[x]$ and $V[c]$, 
$V[c]\models\Phi(z)\in\A_T$.
\end{proof}

\begin{claim}
\label{c-6}
Let $y$ and $z$ be distinct reals from the ground model. Then $\Phi(y)$ and
$\Phi(z)$ are incompatible in $(\A,\subseteq)$.
\end{claim}
\begin{proof}
Fix $y\ne z$ in $\cantor$ from the ground model. Suppose towards a contradiction
that $\Phi(y)$ and $\Phi(z)$ are compatible. Then since there must be a member of
$\A_T$ contained in $\Phi(y)\cap\Phi(z)$,
by equation~\eqref{eq:70}, there some integer $j$ such that
\begin{equation}
  \label{eq:15}
  \bigl|\bigl(\Phi(y)\cap\Phi(z)\bigr)(n_i)\bigr|\ge k(n_i)^\beta\espc\text{for all $i\ge j$}.
\end{equation}

Let $r\in\C$ be a condition forcing~\eqref{eq:15}. By increasing $j$ if necessary,
we assume that 
\begin{equation}
  \label{eq:33}
  y\restriction j\ne z\restriction j.
\end{equation}
By extending $r$ if necessary, we assume that $|r|\ge j$. 
Set $i=|r|$. By claim~\ref{c-5},
\begin{multline}
  \label{eq:36}
  |\shadeto{m(n_i)}\{R(u):u\in\imsucc(b_i(y))\}|\\
  \cdot|\shadeto{m(n_i)}\{R(v):v\in\imsucc(b_i(z))\}|\\
  >N_1(2m(n_i),m(n_i),k(n_i),k(n_i)^\beta).
\end{multline}
Hence $\bigl(\{R(u):u\in\imsucc(b_i(y))\},\{R(v):v\in\imsucc(b_i(z))\}\bigr)$
cannot be a cross-$k(n_i)^\beta$-intersecting pair. 
This means that there are $u\in\imsucc(b_i(y))$ and
$v\in\imsucc(b_i(z))$ such that 
\begin{equation}
  \label{eq:37}
  |R(u)\cap R(v)|<k(n_i)^\beta.
\end{equation}
By~\eqref{eq:33}, we can define an extension $r_0=r\ext g$ of $r$,
where $g\in G(i)$, so that
\begin{equation}
  \label{eq:44}
  b_i(y)\ext g(y\restriction i)\bigl(b_i(y)\bigr)=u
  \And b_i(z)\ext g(z\restriction i)\bigl(b_i(z)\bigr)=v.
\end{equation}
However, $r_0\forces\bigl(\Phi(y)\cap\Phi(z)\bigr)(n_i)=R(u)\cap R(v)$
which by~\eqref{eq:37} is in contradiction with~\eqref{eq:15}.
\end{proof}

Claims~\ref{c-7} and~\ref{c-6} entail that, in the Cohen extension, the image under
$\Phi$ of the ground model reals is an antichain of $(\A,\subseteq)$ of size
continuum. By absoluteness, this implies the existence of an uncountable antichain
of $(\A,\subseteq)$ in the ground model, concluding the paper.
\end{proof}

\bibliographystyle{amsalpha}
\bibliography{database}

\providecommand{\bysame}{\leavevmode\hbox to3em{\hrulefill}\thinspace}
\providecommand{\MR}{\relax\ifhmode\unskip\space\fi MR }
\providecommand{\MRhref}[2]{%
  \href{http://www.ams.org/mathscinet-getitem?mr=#1}{#2}
}
\providecommand{\href}[2]{#2}
\begin{thebibliography}{EKR61}

\bibitem[AK97]{MR1429238}
Rudolf Ahlswede and Levon~H. Khachatrian, \emph{The complete intersection
  theorem for systems of finite sets}, European J. Combin. \textbf{18} (1997),
  no.~2, 125--136. \MR{MR1429238 (97m:05251)}

\bibitem[And02]{MR2003b:05001}
Ian Anderson, \emph{Combinatorics of finite sets}, Dover Publications Inc.,
  Mineola, NY, 2002, Corrected reprint of the 1989 edition. \MR{2003b:05001}

\bibitem[BJP05]{MR2186722}
B.~Balcar, T.~Jech, and T.~Paz{\'a}k, \emph{Complete {CCC} {B}oolean algebras,
  the order sequential topology, and a problem of von {N}eumann}, Bull. London
  Math. Soc. \textbf{37} (2005), no.~6, 885--898. \MR{MR2186722 (2006j:28013)}

\bibitem[EKR61]{MR0140419}
P.~Erd{\H{o}}s, Chao Ko, and R.~Rado, \emph{Intersection theorems for systems
  of finite sets}, Quart. J. Math. Oxford Ser. (2) \textbf{12} (1961),
  313--320. \MR{MR0140419 (25 \#3839)}

\bibitem[Eng97]{MR1429390}
Konrad Engel, \emph{Sperner theory}, Encyclopedia of Mathematics and its
  Applications, vol.~65, Cambridge University Press, Cambridge, 1997.
  \MR{MR1429390 (98m:05187)}

\bibitem[Fra78]{MR519277}
P.~Frankl, \emph{The {E}rd{\H o}s-{K}o-{R}ado theorem is true for {$n=ckt$}},
  Combinatorics (Proc. Fifth Hungarian Colloq., Keszthely, 1976), Vol. I,
  Colloq. Math. Soc. J\'anos Bolyai, vol.~18, North-Holland, Amsterdam, 1978,
  pp.~365--375. \MR{MR519277 (80c:05014)}

\bibitem[Hir08a]{Hirs2}
James Hirschorn, \emph{Asymptotic upper bounds on the shades of
  $t$-intersecting families}, arXiv:0808.1434v1, 2008.

\bibitem[Hir08b]{Hirs}
\bysame, \emph{Nonhomogeneous analytic families of trees \tu(ar{X}iv version
  \tu1\tu)}, arXiv:0807.0147v1, 2008.

\bibitem[Kam98]{MR1610563}
Anastasis Kamburelis, \emph{Dominating analytic families}, Fund. Math.
  \textbf{156} (1998), no.~1, 73--83. \MR{MR1610563 (99i:03056)}

\bibitem[MT89]{MR90g:05008}
Makoto Matsumoto and Norihide Tokushige, \emph{The exact bound in the {E}rd{\H
  o}s-{K}o-{R}ado theorem for cross-intersecting families}, J. Combin. Theory
  Ser. A \textbf{52} (1989), no.~1, 90--97. \MR{90g:05008}

\bibitem[She94]{MR1303493}
Saharon Shelah, \emph{How special are {C}ohen and random forcings, i.e.\
  {B}oolean algebras of the family of subsets of reals modulo meagre or null},
  Israel J. Math. \textbf{88} (1994), no.~1-3, 159--174. \MR{MR1303493
  (96g:03090)}

\bibitem[She01]{MR1832454}
\bysame, \emph{Consistently there is no non trivial ccc forcing notion with the
  {S}acks or {L}aver property}, Combinatorica \textbf{21} (2001), no.~2,
  309--319, Paul Erd\H os and his mathematics (Budapest, 1999). \MR{MR1832454
  (2002k:03072)}

\bibitem[Tal06]{MR2214604}
Michel Talagrand, \emph{Maharam's problem}, C. R. Math. Acad. Sci. Paris
  \textbf{342} (2006), no.~7, 501--503. \MR{MR2214604 (2006k:28006)}

\bibitem[Tod00]{MR1809418}
Stevo Todor{\v{c}}evi{\'c}, \emph{A dichotomy for {P}-ideals of countable
  sets}, Fund. Math. \textbf{166} (2000), no.~3, 251--267. \MR{MR1809418
  (2001k:03111)}

\bibitem[Vel02]{MR1903857}
Boban Velickovic, \emph{The basis problem for {CCC} posets}, Set theory
  (Piscataway, NJ, 1999), DIMACS Ser. Discrete Math. Theoret. Comput. Sci.,
  vol.~58, Amer. Math. Soc., Providence, RI, 2002, pp.~149--160. \MR{MR1903857
  (2003e:03098)}

\bibitem[Vel05]{MR2166361}
\bysame, \emph{ccc forcing and splitting reals}, Israel J. Math. \textbf{147}
  (2005), 209--220. \MR{MR2166361 (2006d:03088)}

\end{thebibliography}

\bigskip
\noindent{\textsc{Thornhill, ON, Canada}}\\
\noindent{\textit{E-mail address}}:
  \href{mailto:j_hirschorn@yahoo.com}{\texttt{j\textunderscore
      hirschorn@yahoo.com}}\\
\noindent{\textit{URL}}:
  \href{http://www.logic.univie.ac.at/~hirschor}
  {\texttt{http://www.logic.univie.ac.at/\textasciitilde hirschor}}

\end{document}